\documentclass[12pt,oneside]{article}

\usepackage{amsmath}
\usepackage{amssymb}
\usepackage{color}
\usepackage{rotating}

\usepackage{graphicx}
\usepackage{subfigure}

\usepackage{url}

\author{Hui Zhou
\\ \footnotesize{Department of Mathematics, Southern University of Science and Technology,}
\\ \footnotesize{Shenzhen, 518055, P.~R.~China}
\\ \footnotesize{School of Mathematical Sciences, Peking University, Beijing, 100871, P.~R.~China}
\\ \footnotesize{zhouhpku17@pku.edu.cn}
\vspace{2em}
\\ Rongquan Feng
\\ \footnotesize{School of Mathematical Sciences, Peking University, Beijing, 100871, P.~R.~China}
\\ \footnotesize{fengrq@math.pku.edu.cn}
}

\title{\Large On distance matrices of distance-regular graphs
}

\def\j{{\mathbf{j}}}
\def\deg{{\sf deg}}

\newtheorem{theorem}{Theorem}
\newtheorem{lemma}{Lemma}
\newtheorem{hypothesis}{Hypothesis}

\newcommand*{\QEDA}{\hfill\ensuremath{\blacksquare}}  
\newenvironment{proof}[1][\hspace{2ex}\textbf{\textit{Proof}.}\hspace{1ex}]{\begin{trivlist}\item[\hskip \labelsep {\bfseries #1}]}{\QEDA\end{trivlist}}

\begin{document}

\maketitle

\begin{abstract}
In this paper, we give a characterization of distance matrices of distance-regular graphs to be invertible.
\end{abstract}





\section{Intrroduction}\label{sec-Introduction}

The study on graphs and matrices is an important topic in algebraic graph theory. The determinant and inverse of the distance matrix of a graph is of great interest. This kind of research is initialed by Graham and Pollak~\cite{Graham Pollak on addressing problem LS} on the determinant of the distance matrix of a tree. Later, Graham, Hoffman and Hosoya~\cite{Graham DM of a directed graph} gave an attractive theorem expressing the determinant of the distance matrix of a graph through that of its blocks, and Graham and Lovasz~\cite{Graham DM Polynomials of trees} calculated the inverse of the distance matrix of a tree.

Recent years, a lot of research has been done on the determinant and the inverse of the distance matrix of a graph, such as: the block graph~\cite{Bapat inv of DM of a block graph}, the odd-cycle-clique graph~\cite{Hou Fang Sun Inverse of cycle-clique graph}, the bi-block graph~\cite{Hou Sun Inverse of bi-block graph}, the multi-block graph~\cite{Zhou2017DMdistance-wellDefined}, the mixed block graph~\cite{ZhouDing2017MixedBlockGraphs}, the weighted tree~\cite{Bapat on DM and Lap}, the bidirected tree~\cite{Bapat bidirected tree}, the arc-weighted tree~\cite{ZhouDing2016DMweightedTree}, the cactoid digraph~\cite{Hou Chen Inverse of cactoid graph}, the weighted cactoid digraph~\cite{ZhouDing2017DMweightedCactoidDigraph}, etc.

Distance-regular graphs are very important in algebraic graph theory, and they have very beautiful combinatorial properties. In this paper, we first give the inverse of the distance matrix of a strongly-regular graph by some tricky calculations in Section~\ref{sec-srg-dm}. Then by the theory of linear systems, we give characterizations of distance matrices of distance-regular graphs to be invertible in Section~\ref{sec-LinearSys}, and by using this result we give the inverse of the distance matrix of a strongly-regular graph in a systematic method.


\section{Preliminaries}\label{sec-preliminaries}

Let $\j$ be an appropriate size column vector whose entries are ones, let $J$ be an appropriate size matrix whose entries are ones, let $\mathbf{0}$ be an appropriate size matrix whose entries are zeroes, and let $I$ be the identity matrix with an appropriate size. For any matrix $M$ and any column vector $\alpha$, we use $M^T$ and $\alpha^T$ to denote their transposes. Let $A=(a_{ij})$ be an $n\times n$ matrix. We denote its determinant by $\det(A)$ and $A_{ij}$ the $(i,j)$-entry $a_{ij}$ of $A$. 

Let $G$ be a connected graph. We use $V(G)$ and $E(G)$ to denote the vertex set and edge set of $G$, respectively. The distance $\partial_G(u,v)$ from vertex $u$ to vertex $v$ in $G$ is the length (number of edges) of the shortest path from $u$ to $v$ in $G$. The distance matrix $D$ of $G$ is a $|V(G)|\times |V(G)|$ square matrix whose $(u,v)$-entry is the distance $\partial_G(u,v)$, that is $D=(\partial_G(u,v))_{u,v\in V(G)}$. Let $v$ be a vertex of $G$. For any integer $i\geqslant 0$, we use $G_i(v)$ to denote the set of vertices $w$ satisfying $\partial_G(v,w)=i$. The degree of $v$ is the number of adjacent vertices, i.e., $\deg_G(v)=|G_1(v)|$. The graph $G$ is called regular with degree $k$, or $k$-regular, if for any vertex $u$ of $G$, the degree $\deg_G(u)=k$. A cut vertex of $G$ is a vertex whose deletion results in a disconnected graph. A block of $G$ is a connected subgraph on at least two vertices such that it has no cut vertices and is maximal with respect to this property.

A complete graph is a graph, in which every pair of vertices are adjacent, and an empty graph is a graph with no edges. Let $G$ be a regular graph, which is neither complete nor empty. Then $G$ is said to be strongly-regular with parameters ($n$, $k$, $a$, $c$), if it is a $k$-regular graph with $n$ vertices, every pair of adjacent vertices has $a$ common neighbours, and every pair of distinct non-adjacent vertices has $c$ common neighbours. 

A connected regular graph $G$ with degree $k\geqslant 1$ and diameter $d\geqslant 1$ is called a distance-regular graph if there are natural numbers \begin{center}$b_0=k,b_1,\ldots,b_{d-1},~c_1=1,c_2,\ldots,c_d,$\end{center} such that for each pair of vertices $u$ and $v$ with $\partial_G(u,v)=j$ we have:
\begin{enumerate}
\item[(1)] the number of vertices in $G_1(u)\cap G_{j-1}(v)$ is $c_j$ ($1\leqslant j\leqslant d$);
\item[(2)] the number of vertices in $G_1(u)\cap G_{j+1}(v)$ is $b_j$ ($0\leqslant j\leqslant d-1$).
\end{enumerate}
The array $\{b_0=k,b_1,\ldots,b_{d-1};c_1=1,c_2,\ldots,c_d\}$ is called the intersection array of $G$, and it is denoted by $\iota(G)$. For $1\leqslant i\leqslant d-1$, let $a_i=k-b_i-c_i$, and let $a_d=k-c_d$. The intersection array is also written as
\begin{equation*}
\iota(G)=\left\{\begin{array}{cccccc}
               *     & c_1=1 & c_2 & \cdots & c_{d-1} & c_d\\
               a_0=0 & a_1 & a_2 & \cdots & a_{d-1} & a_d\\
               b_0=k & b_1 & b_2 & \cdots & b_{d-1} & *
               \end{array}\right\}.
\end{equation*}

Let $G$ be a distance-regular graph with valency $k$ and diameter $d$. Let $A$ be the adjacency matrix of $G$, and let $D$ be the distance matrix of $G$. The $i$-distance matrices $A_i$ $(0\leqslant i\leqslant d)$ of $G$ are defined as follows: \begin{center}$(A_i)_{uv}=1$, if $\partial_G(u,v)=i$; $(A_i)_{uv}=0$, otherwise.\end{center} We have \begin{center}$A_0=I$, $A_1=A$, $A\j=k\j$ and $A_0+A_1+A_2+\ldots +A_d=J$.\end{center} By definition, we know $D=\sum\limits_{i=0}^diA_i=A_1+2A_2+\ldots+dA_d$.


\section{Strongly-regular graphs}\label{sec-srg-dm}

By definitions, we know that a strongly-regular graph is a distance-regular graph with diameter two. In this section, we study the distance matrices of strongly-regular graphs.

Let $G$ be a strongly-regular graph with parameters $(n,k,a,c)$. Let $D$ be the distance matrix of $G$, and let $A$ be the adjacency matrix of $G$. Then \begin{equation}\label{eqn-D=JIA}D=2(J-I)-A\end{equation} (this equation also holds for complete graphs). By Lemma~2.5 in~\cite{AGT-Biggs}, $(A^2)_{uv}$ is the number of walks of length two connecting $u$ and $v$. By counting walks of length two, we have \begin{equation}\label{eqn-A2=cJ}A^2+(c-a)A+(c-k)I=cJ.\end{equation} Thus $D$ is a polynomial of $A$ with degree two and \begin{center}$D=\frac{2}{c}A^2+(1-\frac{2a}{c})A-\frac{2k}{c}I=\alpha(A)$\end{center} where $\alpha(x)=\frac{2}{c}x^2+(1-\frac{2a}{c})x-\frac{2k}{c}=-x-2+\frac{2}{c}[x^2+(c-a)x+(c-k)]$. Since $D=A+2A_2$, we get \begin{center}$A_2=\frac{1}{c}(A^2-aA-kI)$.\end{center} The eigenvalues of $A$ are $k$ of multiplicity $1$, $\theta$ of multiplicity $m_\theta$ and $\tau$ of multiplicity $m_\tau$ where $\theta$ and $\tau$ are the roots of $x^2+(c-a)x+(c-k)=0$ and the multiplicities satisfying $m_\theta+m_\tau=n-1$ and $k+\theta m_\theta+\tau m_\tau=0$ (since the trace is zero). By calculations, we find
\begin{eqnarray*}
\theta &=&\frac{1}{2}\biggl[(a-c)+\sqrt{(a-c)^2+4(k-c)}\biggr],\\
\tau &=&\frac{1}{2}\biggl[(a-c)-\sqrt{(a-c)^2+4(k-c)}\biggr],\\
m_\theta &=&\frac{1}{2}\biggl[(n-1)-\frac{2k+(n-1)(a-c)}{\sqrt{(a-c)^2+4(k-c)}}\biggr],\\
m_\tau &=&\frac{1}{2}\biggl[(n-1)+\frac{2k+(n-1)(a-c)}{\sqrt{(a-c)^2+4(k-c)}}\biggr].
\end{eqnarray*}
Thus the eigenvalues of $D$ are $\alpha(k)=2n-k-2$ of multiplicity $1$, $\alpha(\theta)=-\theta-2$ of multiplicity $m_\theta$ and $\alpha(\tau)=-\tau-2$ of multiplicity $m_\tau$. The determinant of $D$ is $\det(D)=(2n-k-2)(-1)^{n-1}(\theta+2)^{m_\theta}(\tau+2)^{m_\tau}$. The distance matrix $D$ is not invertible if and only if $(\theta+2)(\tau+2)=2a+4-k-c=0$, i.e., \begin{equation}\label{eqn-srg-invertible}k+c=2a+4.\end{equation}

Note that $AJ=JA=kJ$, we have \begin{center}$J(A-kI)=\mathbf{0}$.\end{center} Since $J$ is a polynomial of $A$ with degree two, this means
\begin{eqnarray*}
\beta(A)&=&cJ(A-kI)\\
&=&A^3+(c-a-k)A^2+(c-k-kc+ka)A-k(c-k)I\end{eqnarray*}
is the minimal polynomial of $A$. When $k\neq c$, we have \begin{center}$A^{-1}=\frac{1}{k(c-k)}\Bigl[A^2+(c-k-a)A+(c-k+ka-kc)I\Bigr]$.\end{center} Note that the condition $k=c$ holds if and only if the graph G is the complete multipartite graph $K_{m[b]}$ where $m\geqslant 2$, $b\geqslant 2$ and $mb=n$. Since $A\j=k\j$, by Equation~(\ref{eqn-A2=cJ}), we have \begin{equation}\label{eqn-nkac}k(k-a-1)=c(n-k-1).\end{equation}

Suppose $k+c\neq 2a+4$, then $D$ is invertible. We now find the inverse of $D$. By Equation~(\ref{eqn-nkac}), we have $c(2n-k-2)=k(2k+c-2a-2)$. Let
\begin{eqnarray*}
\lambda &=& k+c-2a-4,\\
\mu &=& 2k+c-2a-2,\\
\delta &=& 2k+c-2a-4,\\
f &=& \frac{\delta}{k\lambda \mu}.
\end{eqnarray*}
(Note that $k\geqslant a+1$ and $c\geqslant 1$, so $2k+c-2a-2>0$. When $k=a+1$, then the graph is the complete graph. So for non-complete graphs, we have $k\geqslant a+2$, and in this case $2k+c-2a-4>0$ which implies $f\neq 0$.) By Equation~(\ref{eqn-D=JIA}), we have
\begin{eqnarray*}
\lambda f c DJ &=& \frac{c\delta}{k\mu}(2J-2I-A)J\\
&=& \frac{c\delta}{k\mu}(2n-k-2)J\\
&=&\delta J.
\end{eqnarray*}
By Equation~(\ref{eqn-A2=cJ}), we have
\begin{eqnarray*}
&&D [A-(2+a-c)I]\\
&=& 2(k+c-a-2)J-\lambda I-[A^2+(c-a)A]+(c-k)I]\\
&=& \delta J-\lambda I.
\end{eqnarray*}
Hence \begin{equation*}\lambda fcDJ+D[(2+a-c)I-A]=\lambda I\end{equation*} which implies
\begin{equation}\label{eqn-inv-SRG}D^{-1}=fcJ+\frac{(2+a-c)I-A}{\lambda}=\frac{c\delta}{k\lambda\mu}J+\frac{2+a-c}{\lambda}I-\frac{1}{\lambda}A.\end{equation}
By the above discussion, we have the following result.

\begin{theorem}\label{thm-srg-inv}
Let $G$ be a strongly-regular graph with parameters $(n,k,a,c)$. Let $D$ be the distance matrix of $G$, and let $A$ be the adjacency matrix of $G$. Then $D$ is invertible if and only if \begin{equation}\label{eqn-srg-invertible-thm}k+c\neq 2a+4,\end{equation} and when $D$ is invertible the inverse is \begin{equation}\label{eqn-srg-inv}D^{-1}=\frac{(2+a-c)I-A}{k+c-2a-4}+\frac{c(2k+c-2a-4)}{k(k+c-2a-4)(2k+c-2a-2)}J.\end{equation}
\end{theorem}

Now we give an example. Let $G$ be the cycle of length $5$. Then $G$ is a strongly-regular graph with parameter $(n,k,a,c)=(5,2,0,1)$. We have $k+c\neq 2a+4$ and $(\lambda,\mu,\delta,f)=(-1,3,1,-\frac{1}{6})$. The adjacency matrix of $G$ is
\begin{center}$A=\left(
     \begin{array}{ccccc}
       0 & 1 & 0 & 0 & 1 \\
       1 & 0 & 1 & 0 & 0 \\
       0 & 1 & 0 & 1 & 0 \\
       0 & 0 & 1 & 0 & 1 \\
       1 & 0 & 0 & 1 & 0 \\
     \end{array}
   \right)$\end{center}
and the distance matrix of $G$ is
\begin{center}$D=2(J-I)-A=\left(
     \begin{array}{ccccc}
       0 & 1 & 2 & 2 & 1 \\
       1 & 0 & 1 & 2 & 2 \\
       2 & 1 & 0 & 1 & 2 \\
       2 & 2 & 1 & 0 & 1 \\
       1 & 2 & 2 & 1 & 0 \\
     \end{array}
   \right).$\end{center}
The inverse of $D$ is
\begin{eqnarray*}D^{-1}&=&\frac{1}{6}\left(
                     \begin{array}{ccccc}
                       -7 & 5 & -1 & -1 & 5 \\
                       5 & -7 & 5 & -1 & -1 \\
                       -1 & 5 & -7 & 5 & -1 \\
                       -1 & -1 & 5 & -7 & 5 \\
                       5 & -1 & -1 & 5 & -7 \\
                     \end{array}
                   \right)\\ &=&-\frac{1}{6}J+A-I=fcJ+\frac{(2+a-c)I-A}{\lambda}.\end{eqnarray*}
This coincides with the formula in Theorem~\ref{thm-srg-inv}.


\section{Distance-regular graphs}\label{sec-drg-dm}

Let $G$ be a distance-regular graph with valency $k$ and diameter $d$. Suppose the intersection array of $G$ is \begin{center}$\iota(G)=\{b_0=k,b_1,\ldots,b_{d-1};c_1=1,c_2,\ldots,c_d\}$.\end{center} Let $A$ be the adjacency matrix of $G$, and let $D$ be the distance matrix of $G$. For $1\leqslant i\leqslant d-1$, let $a_i=k-b_i-c_i$, and let $a_d=k-c_d$. By Lemma~20.6 in~\cite{AGT-Biggs}, we have $AA_0=A$, \begin{equation}\label{eqn-AAi}AA_i=b_{i-1}A_{i-1}+a_iA_i+c_{i+1}A_{i+1}~(1\leqslant i\leqslant d-1)\end{equation} and $AA_d=b_{d-1}A_{d-1}+a_dA_d$.

By convention, we let $b_d=c_0=0$. By simple calculation, we have
\begin{eqnarray*}
AD&=&\sum\limits_{i=1}^{d-1}i(b_{i-1}A_{i-1}+a_iA_i+c_{i+1}A_{i+1})+d(b_{d-1}A_{d-1}+a_dA_d)\\
&=&b_0A_0+\sum\limits_{i=1}^{d-1}[(i-1)c_i+ia_i+(i+1)b_i]A_i+[(d-1)c_d+da_d]A_d\\
&=&b_0A_0+\sum\limits_{i=1}^{d-1}[ki+(b_i-c_i)]A_i+(kd-c_d)A_d\\
&=&kD+\sum\limits_{i=0}^d(b_i-c_i)A_i.
\end{eqnarray*}

Additionally, we let $a_0=0$, $c_{d+1}=b_{d+1}=0$ and $b_{-1}=c_{-1}=0$. For $-1\leqslant i\leqslant d+1$, we let $\delta_i=b_i-c_i$. Then we have  $\delta_{-1}=\delta_{d+1}=0$ and $AD=kD+\sum\limits_{i=0}^d\delta_iA_i$. If we suppose $A_{-1}=A_{d+1}=\mathbf{0}$, then we have \begin{center}$AA_i=b_{i-1}A_{i-1}+a_iA_i+c_{i+1}A_{i+1}$ for $0\leqslant i\leqslant d$.\end{center} Let $X_0=D=\sum\limits_{i=0}^diA_i$ and $X_1=\sum\limits_{i=0}^d\delta_iA_i$. We have $AX_0=kX_0+X_1$. Similarly, we have
\begin{eqnarray*}
AX_1&=&\sum\limits_{i=0}^d\delta_i(b_{i-1}A_{i-1}+a_iA_i+c_{i+1}A_{i+1})\\
&=&\sum\limits_{i=0}^d(\delta_{i-1}c_i+\delta_ia_i+\delta_{i+1}b_i)A_i\\
&=&k\sum\limits_{i=0}^d\delta_iA_i+\sum\limits_{i=0}^d((\delta_{i+1}-\delta_i)b_i-(\delta_i-\delta_{i-1})c_i)A_i.
\end{eqnarray*}
For $0\leqslant i\leqslant d$, let $\zeta_i=(\delta_{i+1}-\delta_i)b_i-(\delta_i-\delta_{i-1})c_i$. Let $X_2=\sum\limits_{i=0}^d\zeta_iA_i$. Then we have $AX_1=kX_1+X_2$. So $A^2X_0=k^2X_0+2kX_1+X_2$. Suppose $\zeta_{-1}=\zeta_{d+1}=0$. By calculation, we have
\begin{eqnarray*}
AX_2&=&\sum\limits_{i=0}^d\zeta_i(b_{i-1}A_{i-1}+a_iA_i+c_{i+1}A_{i+1})\\
&=&\sum\limits_{i=0}^d(\zeta_{i-1}c_i+\zeta_ia_i+\zeta_{i+1}b_i)A_i\\
&=&k\sum\limits_{i=0}^d\zeta_iA_i+\sum\limits_{i=0}^d((\zeta_{i+1}-\zeta_i)b_i-(\zeta_i-\zeta_{i-1})c_i)A_i.
\end{eqnarray*}
For $0\leqslant i\leqslant d$, let $\phi_i=(\zeta_{i+1}-\zeta_i)b_i-(\zeta_i-\zeta_{i-1})c_i$. Let $X_3=\sum\limits_{i=0}^d\phi_iA_i$. Then we have $AX_2=kX_2+X_3$. So $A^3X_0=k^3X_0+3k^2X_1+3kX_2+X_3$.

We can calculate $AX_i$ and $A^iX_0$ for $0\leqslant i\leqslant d$ in a similar way. These properties between $X_i$'s can be extended to a general case in the next section.


\section{Linear systems}\label{sec-LinearSys}

By the properties of $i$-distance matrices ($0\leqslant i\leqslant d$) of distance-regular graphs, we consider the following matrices $A_0,A_1,\ldots,A_d$ satisfying properties the same as $i$-distance matrices of distance-regular graphs. Then we analyze these matrices by linear systems to give characterizations of the matrix $X$ (corresponding to the distance matrix of a distance-regular graph) to be invertible. First, we give some notations in the following hypothesis.

\begin{hypothesis}\label{hypo-linearSys}
Let $d\geqslant 2$ and $k\geqslant 3$.
\begin{enumerate}
\item[(1)] Let
$\iota=\left\{\begin{array}{cccccccc}
               c_{-1}=0 &  c_0=0 & c_1=1 & c_2 & \cdots & c_{d-1} & c_d   & c_{d+1}=0\\
               a_{-1}   &  a_0=0 & a_1   & a_2 & \cdots & a_{d-1} & a_d   & a_{d+1}\\
               b_{-1}=0 &  b_0=k & b_1   & b_2 & \cdots & b_{d-1} & b_d=0 & b_{d+1}=0
               \end{array}\right\}
$ 
be an array satisfying $k=c_i+a_i+b_i$ for $-1\leqslant i\leqslant d+1$ and $1=c_1\leqslant c_2\leqslant\cdots\leqslant c_d\leqslant k$, $k=b_0\geqslant b_1\geqslant \cdots\geqslant b_{d-1}\geqslant 1$.
\item[(2)] Let $A_{-1}=\mathbf{0},A_0,A_1,\ldots,A_d,A_{d+1}=\mathbf{0}$ be a series of matrices satisfying $A_0=I$, $A_1=A$, $A\j=k\j$, the minimal polynomial of $A$ is of degree at least $d+1$, $J=A_0+A_1+\ldots+A_d$ and \begin{equation}\label{eqn-AAi-d}AA_i=b_{i-1}A_{i-1}+a_iA_i+c_{i+1}A_{i+1} \text{ for $0\leqslant i\leqslant d$}.\end{equation}
\item[(3)] Let $\{x_{i,j}\mid 0\leqslant i\leqslant d, ~-1\leqslant j\leqslant d+1\}$ be a series of numbers satisfying \begin{center}$x_{i,-1}=x_{i,d+1}=0$ for $0\leqslant i\leqslant d$, and\end{center} \begin{center}$x_{i+1,j}=(x_{i,j+1}-x_{i,j})b_j-(x_{i,j}-x_{i,j-1})c_j$\end{center} for $0\leqslant i\leqslant d-1$ and $0\leqslant j\leqslant d$.
\item[(4)] Let $X_i=\sum\limits_{j=0}^dx_{i,j}A_j$ for $0\leqslant i\leqslant d$.
\item[(5)] Let $X=X_0=\sum\limits_{j=0}^dx_{0,j}A_j$.
\item[(6)] Let $Q=(x'_{i,j})_{0\leqslant i\leqslant d,~0\leqslant j\leqslant d}$ where $x'_{i,j}=x_{j,i}$ for $0\leqslant i\leqslant d$ and $0\leqslant j\leqslant d$.
\item[(7)] Let $\alpha_{h,i}=\sum\limits_{j=0}^i\binom{i}{j}k^{i-j}x_{j,h}$ for $0\leqslant i\leqslant d$ and $0\leqslant h\leqslant d$.
\end{enumerate}
\end{hypothesis}

Note that, by Proposition~2.6 in~\cite{AGT-Biggs}, here we suppose in (2) of the above hypothesis that the minimal polynomial of $A$ is of degree at least $d+1$.

\begin{lemma}\label{lem-mim-poly-degree-i}
We use notations as in Hypothesis~\ref{hypo-linearSys}. Then $A_i$ is a polynomial in $A$ of degree $i$ for $0\leqslant i\leqslant d$, $J$ is a polynomial $f(A)$ in $A$ of degree $d$, and the minimal polynomial of $A$ is $\mu(A)=\left(\prod\limits_{i=1}^dc_i\right)f(A)(A-kI)$ of degree $d+1$.
\end{lemma}

\begin{proof}
By Equation~(\ref{eqn-AAi-d}), we can find the polynomial of $A_i$ in $A$ recursively for $2\leqslant i\leqslant d$. So $J$ is a polynomial in $A$ of degree $d$. Since $A\j=k\j$, we have $AJ=kJ$. Hence $f(A)(A-kI)=J(A-kI)=\mathbf{0}$. Note the degree of $f(A)(A-kI)$ is $d+1$. Thus the minimal polynomial of $A$ is obtained.
\end{proof}

By Equation~(\ref{eqn-AAi-d}) and Lemma~\ref{lem-mim-poly-degree-i}, we list several $A_i$'s as follows.
\begin{eqnarray*}
f_2(A)=c_2A_2&=&A^2-a_1A-kI,\\
f_3(A)=c_2c_3A_3&=&A^3-(a_1+a_2)A^2+(a_1a_2-b_1c_2-k)A+ka_2I,\\
f_4(A)=c_2c_3c_4A_4&=&A^4-(a_1+a_2+a_3)A^3+(a_1a_2+a_1a_3+a_2a_3-b_1c_2\\
&&-b_2c_3-k)A^2+(ka_2+ka_3+b_1c_2a_3+a_1b_2c_3\\
&&-a_1a_2a_3)A+(kb_2c_3-ka_2a_3)I.
\end{eqnarray*}
We pick the item $f_d(A)=\left(\prod\limits_{i=1}^{d}c_i\right)A_{d}$ in this list. For example, when $d=2$, then $f_3(A)$ is the minimal polynomial of $A$; when $d=3$, then $f_4(A)$ is the minimal polynomial of $A$. We may try to prove that $f_{d+1}(A)$ is the minimal polynomial of $A$ when the diameter is $d$. It is easy to check that $x=k$ is a root of $f_{d+1}(x)$. So all the other eigenvalues of $A$ are the roots of $\beta(x)=\frac{f_{d+1}(x)}{x-k}$. By Theorem~20.7 in~\cite{AGT-Biggs}, a distance-regular graph with diameter $d$ has just $d+1$ distinct eigenvalues which are $x=k$ and the roots of $\beta(x)$.

Using notations as in Section~\ref{sec-drg-dm} and Hypothesis~\ref{hypo-linearSys}, we have \begin{center}$x_{0,i}=i$, $x_{1,i}=\delta_i$, $x_{2,i}=\zeta_i$ and $x_{3,i}=\phi_i$ for $0\leqslant i\leqslant d$;\\ $X_0=D=\sum\limits_{i=0}^diA_i$, $X_1=\sum\limits_{i=0}^d\delta_iA_i$, $X_2=\sum\limits_{i=0}^d\zeta_iA_i$, $X_3=\sum\limits_{i=0}^d\phi_iA_i$;\\ $AX_0=kX_0+X_1$, $AX_1=kX_1+X_2$, $AX_2=kX_2+X_3$;\\ $A^2X_0=k^2X_0+2kX_1+X_2$ and $A^3X_0=k^3X_0+3k^2X_1+3kX_2+X_3$.\end{center} Generally, we have the following result.

\begin{lemma}\label{lem-AXi}
We use notations as in Hypothesis~\ref{hypo-linearSys}. Then
\begin{enumerate}
\item[(1)] $AX_i=kX_i+X_{i+1}$ for $0\leqslant i\leqslant d-1$,
\item[(2)] $A^iX=\sum\limits_{j=0}^i\binom{i}{j}k^{i-j}X_j=\sum\limits_{h=0}^d\alpha_{h,i}A_h$ for $0\leqslant i\leqslant d$.
\end{enumerate}
\end{lemma}

\begin{proof}
\begin{enumerate}
\item[(1)] Let $0\leqslant i\leqslant d-1$. Then
\begin{eqnarray*}
AX_i&=&\sum\limits_{j=0}^dx_{i,j}AA_j\\
&=&\sum\limits_{j=0}^dx_{i,j}(b_{j-1}A_{j-1}+a_jA_j+c_{j+1}A_{j+1})\\
&=&\sum\limits_{j=0}^d(x_{i,j-1}c_j+x_{i,j}a_j+x_{i,j+1}b_j)A_j\\
&=&\sum\limits_{j=0}^d[kx_{i,j}+(x_{i,j+1}-x_{i,j})b_j-(x_{i,j}-x_{i,j-1})c_j]A_j\\
&=&\sum\limits_{j=0}^dkx_{i,j}A_j+\sum\limits_{j=0}^dx_{i+1,j}A_j\\
&=&kX_i+X_{i+1}.
\end{eqnarray*}
\item[(2)] We prove the first equality by induction on $i$. Suppose for $i<d$, we have $A^iX=\sum\limits_{j=0}^i\binom{i}{j}k^{i-j}X_j$. Then

\begin{eqnarray*}
A^{i+1}X&=&\sum\limits_{j=0}^i\binom{i}{j}k^{i-j}(kX_j+X_{j+1})\\
&=&\binom{i}{0}k^{i+1}X_0+\sum\limits_{j=0}^{i-1}\left[\binom{i}{j}+\binom{i}{j+1}\right]k^{i-j}X_{j+1}+\binom{i}{i}k^0X_{i+1}\\
&=&\binom{i+1}{0}k^{i+1}X_0+\sum\limits_{j=0}^{i-1}\binom{i+1}{j+1}k^{(i+1)-(j+1)}X_{j+1}+\binom{i+1}{i+1}k^0X_{i+1}\\
&=&\sum\limits_{j=0}^{i+1}\binom{i+1}{j}k^{i+1-j}X_j.
\end{eqnarray*}

Let $0\leqslant i\leqslant d$. We have \begin{center}$A^iX=\sum\limits_{j=0}^i\binom{i}{j}k^{i-j}X_j
=\sum\limits_{h=0}^d\sum\limits_{j=0}^i\binom{i}{j}k^{i-j}x_{j,h}A_h=\sum\limits_{h=0}^d\alpha_{h,i}A_h$.\end{center}
\end{enumerate}
\end{proof}

By Hypothesis~\ref{hypo-linearSys}, $X$ is a polynomial of $A$. If the matrix $X$ is invertible, by Lemma~\ref{lem-mim-poly-degree-i}, the inverse of $X$ is a polynomial in $A$ of degree at most $d$. Thus the matrix $X$ is invertible if and only if there exist $y_0,y_1,y_2,\ldots,y_d$ such that the following equation holds
\begin{equation}\label{eqn-yiAiX=I}
\sum\limits_{i=0}^dy_iA^iX=I.
\end{equation}
In this case, the inverse is $X^{-1}=\sum\limits_{i=0}^dy_iA^i$. We consider Equation~(\ref{eqn-yiAiX=I}). Let
\begin{equation}\label{eqn-zj}
z_j=\sum\limits_{i=j}^d\binom{i}{j}k^{i-j}y_i
\end{equation}
for $0\leqslant j\leqslant d$. By M\"{o}bius inversion formula, we have \begin{center}$y_i=\sum\limits_{j=i}^d(-1)^{j-i}\binom{j}{i}k^{j-i}z_j$ for $0\leqslant i\leqslant d$.\end{center}

\begin{lemma}\label{lem-eqn-yiAiX-zjxjhAh}
We use notations as in Hypothesis~\ref{hypo-linearSys}. Suppose $y_0,y_1,y_2,\ldots,y_d$ is a solution of Equation~(\ref{eqn-yiAiX=I}), and we let $z_0,z_1,z_2,\ldots,z_d$ be defined as in Equation~(\ref{eqn-zj}). For $0\leqslant h\leqslant d$, let
\begin{eqnarray}
L_h&=&\sum\limits_{i=0}^dy_i\alpha_{h,i},\label{eqn-Lh}\\
M_h&=&\sum\limits_{j=0}^dz_jx_{j,h}.\label{eqn-Mh}
\end{eqnarray}
Then
\begin{equation}\label{eqn-yiAiX-LhAh-MhAh}
\sum\limits_{i=0}^dy_iA^iX=\sum\limits_{h=0}^dL_hA_h=\sum\limits_{h=0}^dM_hA_h.
\end{equation}
\end{lemma}

\begin{proof}
We have
\begin{eqnarray*}
\sum\limits_{i=0}^dy_iA^iX&=&\sum\limits_{i=0}^dy_i\sum\limits_{h=0}^d\alpha_{h,i}A_h
=\sum\limits_{h=0}^d\biggl(\sum\limits_{i=0}^dy_i\alpha_{h,i}\biggr)A_h\\
&=&\sum\limits_{h=0}^d\biggl(\sum\limits_{i=0}^dy_i\sum\limits_{j=0}^i\binom{i}{j}k^{i-j}x_{j,h}\biggr)A_h\\
&=&\sum\limits_{h=0}^d\sum\limits_{j=0}^d\biggl(\sum\limits_{i=j}^d\binom{i}{j}k^{i-j}y_i\biggr)x_{j,h}A_h
=\sum\limits_{h=0}^d\biggl(\sum\limits_{j=0}^dz_jx_{j,h}\biggr)A_h.
\end{eqnarray*}
\end{proof}

By Equation~\ref{eqn-yiAiX-LhAh-MhAh}, we consider the following linear systems. We denote $\mathbf{L}$ the following system of linear equations
\begin{equation*}
\left\{\begin{array}{l}L_0=1, \\ L_i=0~(1\leqslant i\leqslant d)\end{array}\right.
\end{equation*}
in $y_0,y_1,y_2,\ldots,y_d$, so the coefficient matrix of $L$ is $P=(\alpha_{h,i})_{0\leqslant h\leqslant d,~0\leqslant i\leqslant d}$. We denote $\mathbf{M}$ the following system of linear equations
\begin{equation*}
\left\{\begin{array}{l}M_0=1, \\ M_i=0~(1\leqslant i\leqslant d)\end{array}\right.
\end{equation*}
in $z_0,z_1,z_2,\ldots,z_d$, so the coefficient matrix of $M$ is $Q$.

\newpage

By the above discussion and the theory of linear systems, we have the following result.

\begin{theorem}\label{thm-equivalence}
We use notations as in Hypothesis~\ref{hypo-linearSys}. The following statements are equivalent.
\begin{enumerate}
\item[(1)] The matrix $X$ is invertible.
\item[(2)] There exist $y_0,y_1,y_2,\ldots,y_d$ such that Equation~(\ref{eqn-yiAiX=I}) holds.
\item[(3)] The system of linear equations $\mathbf{L}$ has a solution.
\item[(4)] The system of linear equations $\mathbf{M}$ has a solution.
\item[(5)] The matrix $P$ is invertible.
\item[(6)] The matrix $Q$ is invertible.
\item[(7)] The determinant $\det(P)\neq 0$.
\item[(8)] The determinant $\det(Q)\neq 0$.
\end{enumerate}
\end{theorem}

Suppose $\det (Q)\neq 0$. Then by Theorem~\ref{thm-equivalence}, the matrix $X$ is invertible. By Cramer's rule, we have a solution of $\mathbf{M}$, that is $z_j=\frac{\det(Q_j)}{\det(Q)}$ for $0\leqslant j\leqslant d$, where $Q_j$ is the matrix obtained from $Q$ by replacing its $j$-th column by $[1,0,0,\ldots,0]^T$. So
\begin{equation}\label{eqn-yi-Qj}
y_i=\sum\limits_{j=i}^d(-1)^{j-i}\binom{j}{i}k^{j-i}\frac{\det(Q_j)}{\det(Q)} \text{ for $0\leqslant i\leqslant d$}.
\end{equation}
Thus \begin{center}$\sum\limits_{i=0}^d\sum\limits_{j=i}^d(-1)^{j-i}\binom{j}{i}k^{j-i}\det(Q_j)A^iX=\det (Q)I$.\end{center} The inverse of $X$ is $X^{-1}=\sum\limits_{i=0}^dy_iA^i=\sum\limits_{i=0}^d\biggl(\sum\limits_{j=i}^d(-1)^{j-i}\binom{j}{i}k^{j-i}\frac{\det(Q_j)}{\det(Q)}\biggr)A^i$. Thus we have the following result.

\begin{theorem}\label{thm-inv-drg}
Using notations as in Hypothesis~\ref{hypo-linearSys}. Let $Q_j$ be the matrix obtained from $Q$ by replacing its $j$-th column by $[1,0,0,\ldots,0]^T$ for $0\leqslant j\leqslant d$. Then $X$ is invertible if and only if $\det(Q)\neq 0$. Furthermore, suppose $\det (Q)\neq 0$. Let $y_i$ ($0\leqslant i\leqslant d$) be defined as in Equation~(\ref{eqn-yi-Qj}). Then the inverse of $X$ is \begin{center}$X^{-1}=\sum\limits_{i=0}^dy_iA^i=\sum\limits_{i=0}^d\biggl(\sum\limits_{j=i}^d(-1)^{j-i}\binom{j}{i}k^{j-i}\frac{\det(Q_j)}{\det(Q)}\biggr)A^i$.\end{center}
\end{theorem}

By Theorem~\ref{thm-inv-drg}, we will calculate $\det(Q)$ and the inverse matrix (when it exists) of the distance matrix of a strongly-regular graph systematically in the next section.


\section{Strongly-regular graphs revisited}\label{sec-srg-rev}

A strongly-regular graph with parameter $(n,k,a,c)$ is a distance-regular graph with intersection array $\{b_0=k,b_1=k-a-1;c_1=1,c_2=c\}$. Let $D$ be the distance matrix of this distance-regular graph. Let $X_0=D=\sum\limits_{i=0}^2iA_i$. Then $x_{0,j}=j$ for $j=0,1,2$. By simple calculation, we get $x_{i,j}$ for $i=1,2$ and $j=0,1,2$, and so \begin{center}$Q=\left(\begin{array}{cll} 0 & k & -k(a+2)\\ 1 & k-a-2 & (a+2-c-k)(k-a-1)+(a+2)\\ 2 & -c & c(k+c-a-2) \end{array}\right)$.\end{center} Hence
\begin{eqnarray*}
\det(Q) & = & k(4+2a-c-k)(2k+c-2a-2)\\ &=& c(2n-k-2)(4+2a-c-k)=c\alpha(k)\alpha(\theta)\alpha(\tau)\\
\det(Q_0) & = & c(4+2a-c-k),\\
\det(Q_1) & = & -4a-2a^2+4c+3ac-c^2+6k+4ak-3ck-2k^2,\\
\det(Q_2) & = & 4+2a-c-2k.
\end{eqnarray*}
Since
\begin{eqnarray*}
\det(Q)y_0 & = & \det(Q_0)-k\det(Q_1)+k^2\det(Q_2),\\
\det(Q)y_1 & = & \det(Q_1)-2k\det(Q_2),\\
\det(Q)y_2 & = & \det(Q_2),
\end{eqnarray*}
we have
\begin{eqnarray*}
\det(Q)[y_0-(c-k)y_2] & = & k(2+a-c)(2+2a-c-2k),\\
\det(Q)[y_1-(c-a)y_2] & = & k(c+2k-2a-2).
\end{eqnarray*}
So
\begin{eqnarray*}
&&\det(Q)[y_0I+y_1A+y_2A^2] \\
& = & \det(Q)[y_0-(c-k)y_2]I+\det(Q)[y_1-(c-a)y_2]A+\det(Q)y_2cJ\\
& = & k(2+a-c)(2+2a-c-2k)I+k(c+2k-2a-2)A+\det(Q_2)cJ.
\end{eqnarray*}

By the above calculation, we know \begin{center}$\det(Q)\neq 0$ if and only if $4+2a-c-k\neq 0$.\end{center} Suppose $\det(Q)\neq 0$, i.e., $4+2a-c-k\neq 0$. Then the inverse of the distance matrix is
\begin{eqnarray*}
D^{-1} & = & \frac{k(2+a-c)(2+2a-c-2k)I+k(c+2k-2a-2)A+\det(Q_2)cJ}{\det(Q)}\\
& = & \frac{(2+a-c)I-A}{k+c-2a-4}+\frac{c(2k+c-2a-4)}{k(k+c-2a-4)(2k+c-2a-2)}J.
\end{eqnarray*}
This coincides with Equation~(\ref{eqn-srg-inv}) in Section~\ref{sec-srg-dm}.


\section{The determinant of $Q$}\label{sec-detQ}

In this section, we consider the determinant of $Q$. In Section~\ref{sec-srg-rev}, we have
\begin{equation}\label{eqn-detQ-2}\det(Q)=c_2\alpha(k)\alpha(\theta)\alpha(\tau).
\end{equation} When the diameter is three, by calculation we will give a similar formula in the following. Hence we guess the determinant of $Q$ has a formula related to the eigenvalues of the distance matrix of the distance-regular graph.

We suppose the diameter is three, i.e., $d=3$. Let $X_0=D=\sum\limits_{i=0}^3iA_i$. Then $x_{0,j}=j$ for $j=0,1,2,3$. By simple calculation, we get $x_{i,j}$ for $i=1,2,3$ and $j=0,1,2,3$, and so \begin{center}$Q=\left(\begin{array}{clll} 0 & k & (b_1-1-k)k & x_{3,0}\\ 1 & b_1-1  & 1-b_1^2+b_1b_2-b_1c_2+k & x_{3,1}\\ 2 & b_2-c_2 & -b_2^2-c_2+b_1c_2+c_2^2-b_2c_3 & x_{3,2}\\ 3 & -c_3 & c_3(b_2-c_2+c_3) & x_{3,3} \end{array}\right)$,\end{center}
where
\begin{eqnarray*}
x_{3,0}&=&k(1-b_1^2+b_1b_2-b_1c_2+2k-b_1k+k^2),\\
x_{3,1}&=&-1-b_1+b_1^2+b_1^3-b_1b_2-b_1^2b_2-b_1b_2^2+2b_1^2c_2+b_1c_2^2-b_1b_2c_3-2k-k^2,\\
x_{3,2}&=&b_2^3+c_2-b_1^2c_2+b_2c_2+b_2^2c_2+c_2^2-2b_1c_2^2-b_2c_2^2-c_2^3+2b_2^2c_3+b_2c_3^2+c_2k,\\
x_{3,3}&=&c_3(-b_2^2-c_2+b_1c_2+c_2^2-2b_2c_3+c_2c_3-c_3^2).
\end{eqnarray*}

Note that $D=\alpha(A)=\frac{1}{c_2c_3}[3A^3+(2c_3-3a_1-3a_2)A^2+(3a_1a_2+c_2c_3-2a_1c_3-3b_1c_2-3k)A+(3ka_2-2kc_3)I]$. Then \begin{center}$\alpha(k)=\frac{k}{c_2c_3}(c_2c_3+2b_1c_3+3b_1b_2)$.\end{center} Let $f_4(x)$ be defined after Lemma~\ref{lem-mim-poly-degree-i}, i.e., $f_4(x)$ is the minimal polynomial of $A$, and let $\theta$, $\tau$ and $\eta$ be the roots of
\begin{equation*}
\beta(x)=\frac{f_4(x)}{x-k}=x^3+(c_3-a_1-a_2)x^2+(a_1a_2+c_2c_3-a_1c_3-b_1c_2-k)x+(a_2a_3-b_2c_3).
\end{equation*}
We know $\alpha(x)=\frac{3}{c_2c_3}\beta(x)-\frac{1}{c_2}r(x)$ where
\begin{center}$r(x)=x^2+(2c_2-a_1)x+(3c_2-k)$.\end{center}
Then $\alpha(y)=-\frac{1}{c_2}r(y)$ for any root $y$ of $\beta(x)$. Let
\begin{eqnarray*}
\pi &=&-3b_2+3b_1^2b_2+3b_2^2-6b_1b_2^2+6c_2-6b_1c_2+6b_1b_2c_2-3c_2^2-2c_3\\
&&+2b_1^2c_3+5b_2c_3-7b_1b_2c_3-3c_2c_3+5b_1c_2c_3-2b_2c_2c_3+2c_2^2c_3\\
&&+2c_3^2-2b_1c_3^2-c_2c_3^2+2b_1k-2b_1^2k-5b_2k+3b_1b_2k+2b_2^2k+7c_2k\\
&&-5b_1c_2k-2c_2^2k-3c_3k+b_1c_3k+3b_2c_3k-c_2c_3k+c_3^2k+b_1k^2\\
&&-2b_2k^2+2c_2k^2-c_3k^2.
\end{eqnarray*}
By calculation, we have
\begin{equation}\label{eqn-detQ-3}
\det(Q)=-k(c_2c_3+2b_1c_3+3b_1b_2)\pi=c_2^2c_3\alpha(k)\alpha(\theta)\alpha(\tau)\alpha(\eta).
\end{equation}

By Equation~(\ref{eqn-detQ-2}) and Equation~(\ref{eqn-detQ-3}), we may guess
\begin{equation}\label{eqn-detQ}\det(Q)=\left(\prod\limits_{i=1}^dc_i^{d+1-i}\right)\left(\prod\limits_{\text{$\lambda$ is an eigenvalue of $D$}}\lambda\right).\end{equation}

%

%



%
%

\end{document}